\documentclass[11pt,twoside, final]{amsart}
\usepackage{amsmath,amsthm,amscd,amsfonts,amssymb,enumerate}
\usepackage{graphicx}		
\usepackage{color}
\usepackage[colorlinks]{hyperref}

 \newcommand{\R}{\mathbb R}
 
\newcommand{\C}{\mathbb C}

\newtheorem{theorem}{Theorem}[section]
\newtheorem{proposition}[theorem]{Proposition}
\newtheorem{lemma}[theorem]{Lemma}
\newtheorem{corollary}[theorem]{Corollary}
\theoremstyle{definition}

\theoremstyle{remark}
\newtheorem{remark}[theorem]{Remark}
\numberwithin{equation}{section}

\begin{document}

\title[Potential theory associated with the Dunkl Laplacian]{Potential theory associated with the Dunkl Laplacian}

\author[K. Hassine]{Kods Hassine}
\address[Kods Hassine]{Faculty of sciences, Monastir university, Tunisia}
\email{hassinekods@gmail.com}

\maketitle

\begin{abstract}
 The main goal of this paper is to give potential theoretical approach to study the Dunkl Laplacian $\Delta_k$ which is a standard example of differential-difference operators. By introducing the Green kernel relative to $\Delta_k$, we prove that the Dunkl Laplacian generates a balayage space and  we investigate the associated family of harmonic measures. Therefore, by mean of harmonic kernels, we give a characterization of all $\Delta_k$-harmonic functions on large class of  open subsets  $U$ of $\R^d$. We also establish existence and uniqueness result of a solution of the corresponding Dirichlet problem.

\end{abstract}

\section{\bf Introduction}
The purpose of this paper is to develop potential theoretic approach to study
a differential-difference operator namely the Dunkl laplacian $\Delta_k$. Roughly speaking,
the Dunkl Laplacian is a perturbation of the usual Laplacian by term
with differences associated with a finite reflection group $W$ and a multiplicity
function $k$. The operator $\Delta_k$ was introduced by C.F.Dunkl in \cite{dunkl2} in order to
construct orthogonal polynomials and spherical harmonics for measure invariant
under the action of a finite reflection group. So the study of the Dunkl
Laplacian was initially in harmonic analysis. Thereafter, M.Voit and M.R\"{o}sler
showed in \cite{rv} that $\Delta_k$ generates a strong Feller semi-group $(P^k_t)_t$ which reduces
to the classical Brownian semi-group in the case where the function $k$ is
identically vanishing. This fact led several authors to define and develop many
Dunkl theoretic concepts probabilistically, by studying the so called Dunkl
process(see \cite{rv,chyb}).
\par Using the Dunkl semi-group$(P^k_t )_t$, we prove in this paper that $\Delta_k$ generates a balayage space. More precisely we introduce  the set $E_{\Delta_k}$  of all excessive functions relative to $(P_t^k)_t$ and we prove that $\R^d$ together with $E_{\Delta_k}$ form a balayage space.
Notice that the notion of  Balayage spaces provides a potential theory which is as rich as that
of harmonic spaces and it covers  large classes of linear elliptic and parabolic partial differential operators  as well as Riesz potentials, Markov chains on discrete spaces and integro-differential operators. Unlike harmonic spaces, it is well known that in balayage spaces, harmonic measures  for  an open set $U$ may live on the entire complement $U^c$ of $U$ instead of being concentrated on the boundary $\partial U$. For our  balayage space $(\R^d, E_{\Delta_k})$ we prove that the associated harmonic measures are  with compact support on the complement.  Moreover, we establish  a  correspondence between harmonic measures and $\Delta_k$-harmonic functions (i.e $C^2$-functions satisfying $\Delta_ku=0$). This correspondence allows us to investigate the existence of a solution  $u\in C^2\cap C(\overline U)$ to the Dirichlet probelm
\begin{equation}\label{pbd}\left\{\begin{array}{rcl}
\Delta_ku&=&0\quad \textrm{ on~} U,\\
u&=&f\quad\textrm{ on~}\partial U,
\end{array}\right.
\end{equation}
for large class of open sets $U$ and  continuous functions $f$.
\par The main tool used in our approach is the Green kernel $G^k$ which is defined by the integral of $P_t^k$ with respect to $t$.
Among the important properties of $G^k$, we shall prove that for every open  Borel bounded function $f$ with compact support on $\R^d$ the function  $G^kf$  is continuous  on $\R^d$, vanishing at infinity and satisfies  $\Delta_k G^kf=-f$ in the distributional
sense (see Theorem~\ref{distr}). Furthermore we show that  for every Borel non negative function $f$, the function $G^kf$ is excessive. By studying excessive functions, we prove  that the couple $(\R^d, E_{\Delta_k})$ is a balayage space. This fact allows us to introduce the corresponding  family of harmonic kernel $(H_V)_V$.  Combining  properties of the Green operator $G^k$ and results known for standard balayage spaces, we prove that the harmonic measure $H_V$ relative to a bounded  open set $V$ is concentrated in the closure  $\overline {^W\!\!V}$ of the set
$$^W\!\!V := \bigcup_{w\in W}w(V).$$
In particular, if $V$ is  $W$-invariant (i.e. $^W\!\!V=V$), then the harmonic measure relative to $V$ is  supported by the boundary $\partial V$. By mean of  harmonic measure, we establish in  the last section of this paper  a characterisation of  $\Delta_k$-harmonic functions. More precisely, we prove that a continuous function $f$ on a bounded $W$-invariant open set $V$ is $\Delta_k$-harmonic on $V$ if and only if $H_Uf=f$ for every open set $U$ such that $\overline U\subset V$. This characterizations leads to prove that Problem~(\ref{pbd}) admits a unique solution provided $U$ is bounded $W$-invariant and regular.

\par The paper is organized as follows: Basic notions and results on
Dunkl theory are collected in Section 2. These concern in particular the Dunkl Laplacian, the Dunkl kernel and the Dunkl translation. Section 3 is devoted to the study of the Green kernel $G^k$. In Section 4, we give a minimum principle which will be used not only to prove the uniqueness but also the existence of a solution to Problem~(\ref{pbd}). In Section 5 we study excessive functions and we prove that  $(\R^d, E_{\Delta_k})$ is a balayage space. By introducing the corresponding  family of harmonic kernel $(H_V)_V$ we give in Section 6  a characterisation of  $\Delta_k$-harmonic functions in $W$-invariant open sets and we  investigate  the Dirichlet problem~(\ref{pbd}).

\section{Preliminary}
For every subset  $F$ of $\R^d$, let $\mathcal{B}(F)$ be the set of all Borel measurable functions on $F$. Let ${C}(F)$ be the set of all continuous real-valued functions on $F$. We denote by $C_0(\R^d)$ the set of all functions $f\in C(\R^d)$ satisfying $\lim_{|x|\to\infty}f(x)=0$. For every open subset $U$ of $\R^d$, we denote by  $\mathcal{C}_c^\infty(U)$  the set of all infinitely differentiable functions on $U$ with compact support. If $\mathcal {G}$ is a set of numerical
  functions then $\mathcal{G}^+$ (respectively $\mathcal{G}_b$) will denote the class of all functions in $\mathcal{G}$ which
  are nonnegative (respectively bounded). For every open subset $V$ of $\R^d$, we shall write $U\Subset V$ when $U$ is a bounded open set such that $\overline U\subset V$.
\par For every $\alpha\in \R^d$, we denote by $\sigma_\alpha$ the reflection in the hyperplane orthogonal
to $\alpha$. It is given by
$$\sigma_{\alpha}(x):=x-2\frac{\langle x,\alpha\rangle}{|\alpha|^2}\alpha,$$
where $\langle\cdot ,\cdot\rangle$  is the usual inner product on~$\R^d$ and $|\alpha|:=\langle \alpha,\alpha \rangle$.
Let $R$ be a
root system, i.e. a finite subset $R$ of $\R^d\setminus \{0\}$ such that $R\cap \R\alpha = \{\pm\alpha\}$.We shall
denote by $W$ the finite reflection group generated by $\{\sigma_\alpha, \alpha\in R\}$.
Let $k : R \rightarrow \R_+$ be a multiplicity function, i.e $k(w\alpha) = k(\alpha)$ for all $w\in W$ and $\alpha\in R$.
Let $V$ $W$-invariant open set, that is  $w(V) \subset V$ for all $w\in W$. The Dunkl Laplacian $\Delta_k$ is given by
\begin{equation}\label{lapd}
\Delta_kf(x)=\Delta f(x)+\sum_{\alpha\in R}k(\alpha)\left(\frac{\langle\nabla f(x),\alpha\rangle}{\langle\alpha,x\rangle}-\frac{|\alpha|^2}{2}\frac{f(x)-f(\sigma_\alpha(x))}{\langle\alpha,x\rangle^2}\right),
\end{equation}
where $\Delta$ and $\nabla$ denote respectively the usual Laplacian and gradient on $\R^d$.
 A function $f:V\rightarrow \R$ is said to be $\Delta_k$-harmonic on $V$ if $f\in C^2(V)$ and $\Delta_kf=0$ on $V$.
The operator $\Delta_k$ has the following  symmetry property: For   $f\in C^2(V)$ and $\varphi\in C_c^2(V)$
\begin{equation}\label{ti}
\int_{\R^d} \Delta_k f(x) \varphi(x) w_k(x)\,dx=\int_{\R^d} f(x)\Delta_k\varphi(x)w_k(x)\,dx,
\end{equation}
where $w_k$ is the homogeneous weight function defined on~$\R^d$ by
$$
w_k(x)=\prod_{\alpha\in
R}|\langle x,\alpha\rangle|^{k(\alpha)}.
$$
It was proved in \cite{mt2} (see also \cite{kods}) that the operator $\Delta_k$ is hypoelliptic in $W$-invariant open sets, i.e., if $f$ is a continuous function on a $W$-invariant open set $V$ and satisfies
$$\int_{V}f(y)\Delta_k\varphi(y)w_k(y)dy=0\mbox{ for every }\varphi\in C^\infty_c(V)$$
then $f$ is infinitely differentiable on $V$.
 \par According to \cite{dunkl1}, there exists a unique linear isomorphism $V_k$  from the space of homogenous polynomials of degree $n$
on~$\mathbb R^d$ into it self such that  $V_k1=1$ and
$\Delta_k V_k=V_k\Delta$. By [17] the intertwining operator $V_k$ has an homeomorphism
extension to $C^\infty(\R^d)$. The positivity of $V_k$ (see [12]) yields that for every
$x\in \R^d$ there exists a unique probability measures $\mu_x^k$ which is supported by
the convex hull of the orbit of $x$,
$$C(x) := co\{wx, w\in W\}$$
such that
$$V_kf(x) =\int_{\R^d}f(y)d\mu^k_x(y)\quad  \mbox{for every } f\in C^\infty(\R^d).$$
For $x=0$ the measure $\mu_0^k$ is the Dirac measure concentrated at 0 which means that $V_kf(0)=f(0)$.

\par The Dunkl kernel associated with the multiplicity function $k$ is defined on $\R^d\times \R^d$ by
$$E_k(x,y) :=\int_{\R^d} e^{\langle y,\xi \rangle}d\mu_x^k(\xi).$$
It is well known that $E_k$ is positive, symmetric and admits a unique holomorphic extension to $\C^d\times\C^d$ satisfying $E_k(\xi z, w) = E_k(z,\xi w)$ for every $z,w\in \C^d$ and every $\xi\in \C.$ Further, it was proved in  [14] that for each $x\in\R^d$ and $t>0$ there exists a unique compactly
supported probability measure $\sigma_{x,t}^k$ such that
\begin{equation}\label{eq20}
E_k(ix,y) j_{\lambda}(t|y|)=\int_{\R^d}E_k(i\xi,y)d\sigma_{x,t}^k(y)\quad\mbox{ for all }y\in\R^d,
\end{equation}
where $j_\lambda$ is the normalized Bessel function given by
 $$j_\lambda(z)=\Gamma(\lambda+1)\sum_{n=0}^\infty\frac{(-1)^nz^{2n}}{2^{2n}n!\Gamma(n+\lambda+1)}.$$
Moreover,
$${\mbox{ supp}}\sigma_{x,r}^k\subset\bigcup_{w\in W}\overline B(wx,r)\setminus B(0, ||x|-r|).$$

\par The Dunkl translation is defined for every function $f\in C^\infty (\R^d)$ and every $x,y\in \R^d$ by
$$\tau_xf(y):=\int_{\R^d} \int_{\R^d} V_k^{-1}f(\eta+\xi)d\mu^k_x(\xi)d\mu_y^k(\eta),$$
where $V_k^{-1}$ is the inverse of $V_k$ on $C^\infty(\R^d)$. In the
particular case where $f$ is in the Schwartz space $S(\R^d)$ and is radially symmetric (that is there  exists a function $F:\R_+\rightarrow \R$ such that $f = F(|\cdot|)$) then $\tau_xf$ is given by
\begin{equation}\label{eq6}
  \tau_xf(y)=\int_{\R^d}F(\sqrt{|x|^2+|y|^2+2\langle x,\xi \rangle})d\mu_y^k(\xi).
\end{equation}
 Notice that for every $f\in C^\infty(\R^d)$ the map $(x,y)\mapsto \tau_xf(y)$ is symmetric, infinitely differentiable on $\R^d\times\R^d$ and satisfies
 $$\Delta_k\tau_xf=\tau_x\Delta_kf\quad \mbox{ and }\tau_xf(0)=f(0).$$
 Further, if the function $f$ is with compact support then $\tau_xf$ is also with compact support. For arbitrary functions $f,g\in S(\R^d)$, it was proved in \cite{th} that
 \begin{equation}\label{eq10}
 \int_{\R^d}\tau_{-x}f(y)g(y)w_k(y)dy=\int_{\R^d}f(y)\tau_xg(y)w_k(y)dy.
 \end{equation}
  \par Notice that   if the multiplicity function  $k$ vanishes identically then  the Dunkl Laplacian reduces to the classical Laplacian $\Delta$. In this case the measure $\mu_x^k$ is the Dirac measure concentrated at $x$. Then the intertwining operators $V_k$ is the identity operator and so $E_k$  and $\tau_x$ reduces to the classical exponential  function and translation operator respectively.  ThroughoutR this paper we assume that
 $$\lambda:=\frac 12\sum_{\alpha\in R}k(\alpha)+\frac d2-1>0.$$

  According to [14], for every $x\in\R^d$, $r>0$ and $f\in C^\infty(\R^d)$
  $$\frac{1}{d_k}\int_{S^{d-1}} \tau_xf(ry)w_k(y)d\sigma(y)=\int_{\R^d}f(y)d\sigma_{x,r}^k(y),$$
  where $S^{d-1}$ is the unit sphere in $\R^d$, $\sigma$ is the surface area measure on $S^{d-1}$, and $d_k$ is the normalizing constant given by
  $$d_k:=\int_{S^{d-1}}w_k(y)d\sigma(y).$$
  \par In the sequel we write
  $$ M_{x,r}(f):=\int_{\R^d}f(y)d\sigma_{x,r}^k(y),$$
  when the integral makes sense. It was shown in \cite{kods} that for every locally bounded function $g$ and every radial function $f\in S(\R^d)$ with $f= F(|\cdot|)$
  \begin{equation}\label{eq1}
  \int_{\R^d}\tau_{-x}f(y)g(y)w_k(y)dy=d_k\int_0^\infty F(s)s^{2\lambda+1}M_{x,s}(g)ds.
  \end{equation}
 The following result was also proved in \cite{kods}.
 \begin{proposition}\label{prop1}
 Let $V$ be a $W$-invariant open set and let $f$ be a locally bounded function on $V$. Then
  $f\in C^2(V)$ and  $\Delta_kf=0$ on $V$ if and only if $M_{x,t}(f)=f(x)$ for every $x\in V$ and $t>0$ such that $B(x,t)\Subset V$.
 \end{proposition}

\section{Green kernel}
For every Borel measurable  function $f$ on $\R^d$ we define
$$P_t^k f(x)=\int_{\R^d}p_t^k(x,y)f(y)w_k(y)dy\quad x\in\R^d,$$
provided the integral makes sense. Here,
$$p_t^k(x,y)=\tau_{-x}q_t(y)\quad\mbox{where}\quad  q_t(y)=\frac{2 \exp\left(-\frac{|y|^2}{4t}\right)}{d_k (4t)^{\lambda+1}\Gamma(\lambda+1)}.$$
Obviously $p_s^k$ is symmetric and positive. Moreover, by \cite{rosler2}, for every $x,y\in\R^d$ and $t,s>0$ we have the following properties
 \begin{enumerate}
 \item $\int_{\R^d}p_t^k(x,\xi)w_k(\xi)d\xi=1$\\
 \item $\int_{\R^d} p_t^k(z,x)p_s^k(z,y)w_k(z)dz=p_{t+s}^k(x,y)$\\
 \item $p_t^k(x,y)\leq \frac{2 \exp\left( -\frac{(|x|-|y|)^2}{4t}\right)}{d_k(4t)^{\lambda+1}\Gamma(\lambda+1)}.$\\
 \end{enumerate}
These yield that the family $(P_t^k)_{t>0}$ forms a semi group (i.e.  $P_t^k\circ P_s^k=P_{t+s}^k$ for every $t,s>0$)  such that for every $t>0$, the kernel  $P_t^k$ is   Markovien (i.e. $P_t^k1=1$) and strong feller (i.e. $P_t^k(\mathcal B_b(\R^d))\subset C_b(\R^d)$).
 Furthermore, in virtue of  \cite{rv},
$$ \lim_{t\rightarrow 0}||P_t^kf-f||_\infty=0 \quad \mbox{for every }f\in C_0(\R^d).$$
 \par The Green operator $G^k$  will play an important role in our approach. It is defined for every Borel bounded or non negative function $f$ by
$$G^kf(x)=\int_0^\infty P_t^kf(x)dt,\quad x\in\R^d.$$

\begin{proposition}\label{prop8}
Let $\varrho$ be a non negative Borel function  on $\R^d$. Then for every $t>0$ and $x\in\R^d$
\begin{equation}\label{lemme1}
P_t^kG^k\varrho(x)= G^kP_t^k\varrho(x)\quad\mbox{ and }\quad G^k\varrho(x)=\int_0^tP_s^k\varrho(x) ds+ P_t^kG^k\varrho(x).
\end{equation}
In particular, $\lim_{t\rightarrow \infty}P_t^kG^k\varrho(x)=0$ provided $G^k\varrho(x)<\infty$.
\end{proposition}
\begin{proof} Let $t>0$ and $x\in\R^d$. Then
 \begin{eqnarray*}
 P_t^k G^k\varrho(x)&=&\int_{\R^d} G^k\varrho(y)p_t^k(x,y)w_k(y)dy\\
 &=&\int_0^\infty\int_{\R^d}P_s^k\varrho(y)p_t^k(x,y)w_k(y)dyds\\
 &=&\int_0^\infty P_t^kP_s^k\varrho(x)ds= \int_0^\infty P_s^kP_t^k\varrho(x)ds=G^kP_t^k\varrho(x)\\
&=&\int_0^\infty P_{t+s}^k\varrho(x)ds\\
&=&\int_t^\infty P_s^k\varrho(x)ds.
\end{eqnarray*}
 Whence
 \begin{eqnarray*}
 G^k\varrho(x)&=&\int_0^\infty P_s^k\varrho(x)ds\\
 &=&\int_0^t P_s^k\varrho(x)ds+\int_t^\infty P_s^k\varrho(x)ds\\
 &=&\int_0^tP_s^k\varrho(x) ds+ P_t^kG^k\varrho(x).
 \end{eqnarray*}

 \end{proof}

For every $x,y\in \R^d$ we define the Green function by
 $$G^k(x,y)=\int_0^\infty p_t^k(x,y)dt.$$
Obviously, $G^k$ is symmetric and positive on $\R^d\times\R^d$.

\begin{lemma}
 For every $x,y\in\R^d$
 \begin{equation}\label{eq8}
 G^k(x,y)\leq \frac{1}{2d_k\lambda}\left(\min_{w\in W}|wy-x|\right)^{-2\lambda}.
 \end{equation}
 \end{lemma}
 \begin{proof}
  Using~(\ref{eq6}), we see that
  $$G^k(x,y)=\frac{2}{d_k\Gamma(\lambda+1)}\int_{\R^d}\int_0^\infty\frac{1}{(4t)^{\lambda+1}}\exp\left(-\frac{|x|^2+|y|^2-2\langle x,\xi\rangle}{4t}\right)dt d\mu_y^k(\xi).$$
  we make the substitution $t\mapsto \frac{|x|^2+|y|^2-2\langle x,\xi\rangle}{4t}$ to obtain,
  \begin{equation}\label{eq7}
  G^k(x,y)=\frac{1}{2d_k\lambda}\int_{\R^d}(|x|^2+|y|^2-2\langle x,\xi\rangle)^{-\lambda}d\mu^k_y(\xi).
  \end{equation}
  Finally, recall that the support of $\mu_y^k$ is contained in $C(y)$ and observe that for every $\xi\in C(y)$,
  $$|x|^2+|y|^2-2\langle x,\xi\rangle\geq \min_{w\in W}|wy-x|^2$$
 to conclude.
 \end{proof}

 \begin{proposition}\label{prop7}
 Let $f$ be a bounded Borel measurable function on $\R^d$ with compact support. Then $G^kf\in C_0(\R^d)$ and for every $x\in\R^d$
 \begin{equation}\label{eq3}
 G^kf(x)=\int_{\R^d} G^k(x,y)f(y)w_k(y)dy.
 \end{equation}
 \end{proposition}
 \begin{proof} Let $x\in\R^d$. Then
 $$G^kf(x)=\int_0^\infty\int_{\R^d}\tau_{-x}q_t(y)f(y)w_k(y)dy dt.$$
In order to prove (\ref{eq3}), we only have to make sure that we may interchange the order of integration.
   It follows from (\ref{eq1}) that for every  $t>0$,
 $$\int_{\R^d}\tau_{-x}q_t(y)|f(y)|w_k(y)dy=d_k\int_0^{\infty}q_t(s)s^{2\lambda+1}M_{x,s}(|f|)ds$$
 where $q_t(s)=\frac{2 \exp\left(-\frac{s^2}{4t}\right)}{d_k (4t)^{\lambda+1}\Gamma(\lambda+1)}.$
 Let $r>0$ such that $\mbox{supp } f\subset B(0,r)$ and let $c:=\sup_{y\in\R^d}|f(y)|$. Since $$\mbox {supp }\sigma_{x,s}^k\subset \R^d\setminus B(0,|s-|x||),$$ then $M_{x,s}(|f|)=0$ whenever $s>|x|+r$. Hence,
 \begin{eqnarray}
 \int_{\R^d}\tau_{-x}q_t(y)|f(y)|w_k(y)dy&=&d_k\int_0^{r+|x|}q_t(s)s^{2\lambda+1}M_{x,s}(|f|)ds\nonumber\\
 &\leq& d_k c\int_0^{r+|x|}q_t(s)s^{2\lambda+1}ds.\label{eq4}
 \end{eqnarray}
 Thus,
 \begin{eqnarray*}
 \int_0^\infty\int_{\R^d} \tau_{-x}q_t(y)|f(y)|w_k(y)dydt&\leq& d_kc\int_0^{r+|x|}\left(\int_0^\infty q_t(s)dt\right)s^{2\lambda+1}ds\\
 &\leq &\frac{c(r+|x|)^2}{4\lambda}<\infty.
 \end{eqnarray*}
 Hence, we apply Fubini-Tonelli theorem to get~(\ref{eq3}). Now we turn to prove the continuity of $G^kf$. The function $P_t^kf$ is continuous on $\R^d$. Further, by~(\ref{eq4}), for every $R>0$ and $x\in B(0,R)$
 $$|P_t^kf(x)|\leq d_k c\int_0^{r+R}q_t(s)s^{2\lambda+1}ds=:h(t).$$
 By direct computation we see that
 $$\int_0^\infty h(t) dt=\frac{c(r+R)^2}{4\lambda}<\infty.$$
 Thus, by Lebesgue theorem, $G^kf$ is continuous on $B(0,R)$ and then on $\R^d$, since $R$ is arbitrary.
Finally, by (\ref{eq8}), for every $x\in\R^d$ such that $|x|\geq 2r$,
 $$|G^k(x,y)|\leq \frac{1}{2d_k\lambda(|x|-|y|)^{2\lambda}}\leq \frac{1}{2d_k\lambda}(2r-|y|)^{2\lambda}.$$
 Thus $\lim_{|x|\rightarrow \infty}G^k(x,y)=0$ which leads by (\ref{eq3}) to $$\lim_{|x|\rightarrow \infty }G^kf(x)=0.$$
 \end{proof}
 \begin{theorem}\label{distr}
 Let $f\in\mathcal B_b(\R^d)$ with compact support. For every~$\varphi\in C_c^\infty(\R^d)$
 $$\int_{\R^d}G^kf(x)\Delta_k\varphi(x)w_k(x)dx=-\int_{\R^d}f(x)\varphi(x)w_k(x)dx.$$
 \end{theorem}
 \begin{proof}
 Let $$g^k(y):=\int_0^\infty q_t(y)dt=\frac{1}{2d_k\lambda|y|^{2\lambda}}.$$
 It was shown in \cite[Theorem 1.4.4]{mt1} that for every $\varphi\in C_c^\infty(\R^d)$
 $$\int_{\R^d}g^k(y)\Delta_k\varphi(y)w_k(y)dy=-\varphi(0).$$
 This leads to
 \begin{equation}\label{eq9}
G^k(\Delta_k\varphi)(y)=-\varphi(y)
 \end{equation}
 Indeed, using (\ref{eq10}) and the fact that $\Delta_k\tau_z=\tau_z\Delta_k$ we get
 \begin{eqnarray*}
 G^k(\Delta_k\varphi)(y)&=&\int_{\R^d} G^k(x,y)\Delta_k\varphi(x) w_k(x)dx\\
 &=&\int_{\R^d}\int_0^\infty \tau_{-y}q_t(x)\Delta_k\varphi(x)w_k(x)dtdx\\
 &=&\int_0^\infty\int_{\R^d}q_t(x)\Delta_k\tau_y\varphi(x)w_k(x)dxdt\\
 &=&\int_{\R^d} g^k(x)\Delta_k\tau_y\varphi(x)w_k(x)dx\\
 &=&-\tau_{y}\varphi(0)=-\varphi(y).
 \end{eqnarray*}
Whence
 \begin{eqnarray*}
 \int_{\R^d}G^kf(x)\Delta_k\varphi(x)w_k(x)dx&=&\int_{\R^d}\left(\int_{\R^d}G^k(x,y)\Delta_k\varphi(x)w_k(x)dx\right)f(y)w_k(y)dy\\
 &=&-\int_{\R^d}f(y)\varphi(y)w_k(y)dy.
 \end{eqnarray*}
 \end{proof}

 \begin{corollary}
 Let $V$ be a $W$-invariant open set and $f$ be a Borel bounded function on $\R^d$   with compact support such that $f=0$ on $V$. Then
 \begin{equation}\label{cor2}
M_{x,t}(G^kf)=G^kf(x),\quad \mbox{ for every }B(x,t)\Subset V.
\end{equation}
 \end{corollary}
\begin{proof}
In virtue of the above theorem
$$\int_{\R^d}G^kf(x)\Delta_k\varphi(x)w_k(x)dx=0,\quad \mbox{for every }\varphi\in C^\infty_c(V).$$
Hence, by the hypoellipticity of $\Delta_k$ it follows that $G^kf$ is infinitely differentiable on $V$ which yields by~(\ref{ti}) that
$$\int_{\R^d}\Delta_kG^kf(x)\varphi(x)w_k(x)dx=0,\quad \mbox{for every }\varphi\in C^\infty_c(V).$$
This means that $G^kf$ is $\Delta_k$-harmonic on $V$. Finally use Proposition~\ref{prop1} to conclude.
\end{proof}
 \begin{proposition}\label{prop112} Let $y,x\in \R^d$ and $t>0$ and denote $G^k_y:=G^k(\cdot,y)$. Then
 $$M_{x,t}(G^k_y)\leq G^k(x,y).$$
 Moreover, if $B(x,t)\Subset \R^d\setminus O(y)$, where
 $$O(y):=\{wy :w\in W\},$$ then
 $M_{x,t}(G^k_y)=G^k(x,y).$
 \end{proposition}
\begin{proof}
 To abbreviate the notation we denote
 $$v_{x,y}(\xi):=\sqrt{|x|^2+|y|^2-2\langle x,\xi\rangle}.$$
 Using  (\ref{eq7}) we see that
 $$G^k(x,y)= \frac{1}{2\lambda d_k}\int_{\R^d}v_{x,y}(\xi)^{-2\lambda}d\mu_y^k(\xi).$$
 On the other hand, it was shown (see Proof of Theorem 3.1 in \cite{rosler2}) that for every $s>0$,
 \begin{equation}\label{eq12}
 M_{x,t}(p_s^k(\cdot,y))=c_\lambda\int_{\R^d}\int_0^\infty j_{\lambda}(rv_{x,y}(\xi))j_\lambda(rt)e^{-sr^2}r^{2\lambda+1}dr d\mu_y^k(\xi),
 \end{equation}
 where $c_\lambda=\frac{1}{d_k4^\lambda(\Gamma(\lambda+1))^2}$. We integrate (\ref{eq12}) over $\{0<s<\infty\}$ to obtain,
 $$M_{x,t}(G^k_y) =c_\lambda\int_{\R^d}\int_0^\infty j_\lambda(rv_{x,y}(\xi))j_\lambda(rt)r^{2\lambda-1}dr d\mu_y^k(\xi).$$
 Using formula 11.4.33 in \cite{abramowitz} we obtain
 \begin{eqnarray*}
 M_{x,t}(G^k_y)&=&\frac{1}{2\lambda d_k}\int_{\R^d}(\max(t,v_{x,y}(\xi))^{-2\lambda}d\mu_y^k(\xi).
 \end{eqnarray*}
 Hence
 $$M_{x,t}(G^k_y) \leq\frac{1}{2\lambda d_k}\int_{\R^d}v_{x,y}(\xi)^{-2\lambda}d\mu_y^k(\xi)=G^k(x,y).$$
 Moreover, it is easy to see that if $B(x,t)\Subset \R^d\setminus O(y)$ then  $|x-wy|>t$ for every $w\in W$  and so
$$v_{x,y}(\xi)>t \mbox{ for every  }\xi\in C(y).$$
Consequently, for every $B(x,t)\Subset \R^d\setminus O(y)$
 \begin{eqnarray*}
 M_{x,t}(G^k_y)&=&\frac{1}{2\lambda d_k}\int_{\R^d}(v_{x,y}(\xi))^{-2\lambda}d\mu_y^k(\xi)=G^k(x,y).
 \end{eqnarray*}
 \end{proof}

 \section{Minimum principle}
 \begin{lemma}\label{lemme3}
 The set $\{G^k\varphi, \quad \varphi\in \mathcal B^+(\R^d)\}$ is linearly separating. That is, for all $\lambda>0$ and all $x,y\in\R^d$ such that $x \neq y$, there exists $\varphi\in\mathcal B^+(\R^d)$ such that $G^k\varphi(x)\neq \lambda G^k\varphi(y).$
 \end{lemma}
 \begin{proof}
 Let $\lambda>0$ and $x_1,x_2\in\R^d$ such that $x_1\neq x_2$. Let $f \in C_c(\R^d)$ be  a non negative
function such that $f(x_1)\neq \lambda  f(x_2)$. Since $lim_{t\rightarrow 0}P_t^kf=f$ then there exists $t_0>0$ such that
$$P_s^kf(x)\neq \lambda P_s ^kf(y)\quad  \mbox{ for all } 0 < s < t_0.$$
Moreover, it follows from~(\ref{lemme1}) that for every $z\in\R^d$
\begin{equation}\label{eq19}
G^kf(z)=  G^kP_{t_0}^kf(z)+\int_0^{t_0}P_s^kf(z) ds.\end{equation}
 By Proposition~\ref{prop7}  we see that $G^kf$  is finite on $\R^d$. Whence, (\ref{eq19}) yields that either $G^kf(x)\neq \lambda G^kf(y)$  or $G^kP_{t_0}^kf(x)\neq \lambda G^kP_{t_0}^kf(y)$ .
  \end{proof}
 The following lemma follows from (\ref{eq3}) and Proposition~\ref{prop112}.
 \begin{lemma}\label{lemme4}Let $\varphi$ be a non negative  Borel function on $\R^d$. Then
  $$M_{x,t}(G^k\varphi)\leq G^k\varphi(x),\quad \mbox{for every }x\in\R^d  \mbox{ and }t>0.$$
 \end{lemma}

 \begin{theorem}
 \label{prop3}
 Let $\Omega$ be a $W$-invariant bounded open set and let $f$ be a lower semi-continuous function on $\Omega$. Assume that:
 \begin{itemize}
 \item [(a)] For every $z\in\partial \Omega$, $\liminf_{x\rightarrow z}f(x)\geq 0$.
 \item [(b)]For every $x\in \Omega$ and $t>0$ such that $B(x,t)\Subset \Omega$,
 $$M_{x,t}(f)=\int_{\Omega}f(y)d\sigma_{x,t}^k(y)\leq f(x).$$
 \end{itemize}
 Then $f\geq 0$ on $\Omega$.
 \end{theorem}
 \begin{proof}
 We extend $f$ to a lower semi-continuous function $u$ on $\overline \Omega$ by setting $u=f$ on $\Omega$ and $u(z)=\liminf_{x\rightarrow z}f(x)$ for every $z\in\partial \Omega$. Thus $u\geq 0$ on $\partial \Omega$. Let $\alpha=\inf_{x\in\overline \Omega}u(x)$ and $$K=\{x\in\overline \Omega \quad\mbox{such that } u(x)=\alpha\}.$$
 The set $K$ is not empty because $u$ is lower semi-continuous on the compact set $\overline \Omega$. If $K\cap \partial \Omega \neq \emptyset$ then $\alpha\geq 0$ and so for every $x\in \Omega$, $f(x)=u(x)\geq \alpha\geq 0$.
   Suppose now  that $K\cap\partial \Omega=\emptyset$. Then
   $$K=\{x\in\Omega \quad\mbox{such that } f(x)=\alpha\}.$$
Thus,    for every $x\in K$ there exists $t>0$ such that $B(x,t)\Subset \Omega$ and
  $$\alpha=\int_{\Omega}\alpha d\sigma_{x,t}^k(y)\leq \int_{\Omega}f(y)d\sigma_{x,t}^k(y)\leq f(x)=\alpha.$$
  This means that $\int_{\Omega}(f(y)-\alpha)d\sigma_{x,t}^k(y)=0$
  and consequently $$\sigma_{x,t}^k(K)=1.$$
 Let $\mathcal A$ be the set of all non empty compact subsets $A$ of $\R^d$ such that for every $x\in A$ there exists $t>0$ such that $\sigma_{x,t}^k(A)=1$.
 Clearly $K\in\mathcal A$ and  the set $\mathcal A$ is inductively ordered by the converse inclusion
relation. Hence, by Zorn's lemma,  there exists a minimal set $M\in\mathcal A$ such that $K\supset M$. The set  $M$ contains more than one point because $\sigma_{x,t}^k\neq \delta_x$ (see~(\ref{eq20}))and $\sigma_{x,t}^k(M)=1$   for some $x\in M$ and $t>0$. Then by Lemma~\ref{lemme3} there exists a Borel function $\varphi\in\mathcal B^+(\R^d)$ such that the restriction of $G^k\varphi$  on $M$ is non constant. Let us consider the set
$$ M'=\{x\in M \mbox{ such that } G^k\varphi(x)=\beta\},$$
where $\beta=\inf_{x\in M}G^k\varphi(x)$. Then $M'$ is a non empty compact set (since the function $G^k\varphi$ is lower semi continuous on $\R^d$ by Fatou's Lemma) and $M\supsetneq M'$. Furthermore, let $x\in M'$ and  $t>0$ such that $\overline B(x,t)\subset \Omega$. Then
$$\beta=\int_{M}\beta d\sigma_{x,t}^k(y)\leq \int_{M}G^k\varphi(y)d\sigma_{x,t}^k(y)\leq \int_{\Omega}G^k\varphi(y)d\sigma_{x,t}^k(y).$$
We then deduce, in virtue of  Lemma~\ref{lemme4},  that
$$\beta\leq \int_{\Omega}G^k\varphi(y)d\sigma_{x,t}^k(y)\leq G^k\varphi(x)\leq \beta$$
which implies that $\int_{\Omega}(G^k\varphi(y)-\beta)d\sigma_{x,t}^k(y)=0$ and then $\sigma_{x,t}^k(M')=1$. Consequently, $M'\in\mathcal A$ contradicting the minimality of $M$.

 \end{proof}
 \section{Excessive functions and Balayage space}
  A function $f\in\mathcal B^+(\R^d)$ is said to be excessive if $\sup_{t>0}P_t^kf=f$.
  The set of all excessive functions will be denoted by $E_{\Delta_k}$. Obviously, the constant function~$1$ belongs to $E_{\Delta_k}$.
   Notice that  if $f$ is a Borel non negative function such that $P_t^kf\leq f$ for every $t>0$ then
   $$P_t^kf=P_s^kP_{t-s}^kf\leq P_s^kf\quad \mbox{ for every }0<s<t$$
   which means that the map $t\mapsto P_t^kf$ is decreasing on $]0,\infty[$. This yields that
   $$E_{\Delta_k}=\{f\in\mathcal B^+(\R^d) : P_t^kf\leq f\mbox{ for every }t>0\mbox{ and }\lim_{t\rightarrow 0}P_t^kf=f\}.$$
 Then, for every $y\in \R^d$ the function $G^k(\cdot,y)$ is excessive. Moreover, it follows from~(\ref{lemme1}) that for every non negative Borel function $f$  on $\R^d$, $G^kf$ is also excessive.
 \begin{proposition}\label{prop9} Let $u\in\mathcal B^+(\R^d)$. Then $u$ is excessive if and only if  there  exists a sequence $(f_n)_n$  in $\mathcal B^+(\R^d)$ such that $(G^kf_n)_n$ increase to $u$.
  \end{proposition}
  \begin{proof} Assume that there exists a non negative sequence $(f_n)_n$ such that $(G^kf_n)_n$  increases to $u$. Then the fact that for every $n$ the function $G^kf_n$ is excessive yields, in  view of the monotone convergence theorem, that $u$ is also excessive. Conversely, assume that $u$ is excessive.
    Let $v\in\mathcal B^+(\R^d)$ such that $0<G^kv<\infty$ (by Proposition~\ref{prop7} such function may exist). For every $n\geq 1$, we define  $$u_n=\min\{u,n,nG^kv\}\quad \mbox{ and } f_n=n(u_n-P_{\frac1n}u_n).$$
  Of course, the sequence $(u_n)_n$ belongs to $\mathcal B^+(\R^d)$ and   increases to $u$. Thus for every $m\geq 1$,  $(P_{\frac1m}u_n)_n$ is an increasing sequence in $\mathcal B^+(\R^d)$.  Further, $u,n$ and $G^kv$ are excessive. So  $(u_n)_n\subset E_{\Delta_k}$ and hence for every $n\geq 1$
    \begin{equation}\label{eq18}
    P^k_{s}u_n=P^k_{t}P^k_{s-t}u_n\leq P^k_{t}u_n\quad\mbox{for every } {0<t<s}.
    \end{equation}
This means that   for every $n\geq1$, $(P_{\frac1m}u_n)_m$ is an increasing sequence  and then
$$\lim_{n\rightarrow \infty}P^k_{\frac 1n}u_n=\lim_{m\rightarrow\infty}\lim_{n\rightarrow\infty}P^k_{\frac{1}{m}}u_n=\lim_{m\rightarrow\infty}P^k_{\frac{1}{m}}u=u.$$
Hence
\begin{equation}\label{eq17}
\lim_{n\rightarrow\infty }n\int_0^{\frac1n}P_s^ku_nds=u
\end{equation}
because, by~(\ref{eq18}),
$$P^k_{\frac1n}u_n\leq n\int_0^{\frac1n}P_s^ku_nds\leq u_n.$$
Obviously, the proof is finished once we have shown that $G^kf_n= n\int_0^{\frac1n}P_s^ku_nds$.
Let $t>0$. Then,
  \begin{eqnarray}
\int_0^{t} P_s^kf_nds&=&n \left(\int_0^{t} P_s^ku_nds-\int_0^t P_{s+\frac1n}^ku_nds\right)\nonumber\\
  &=&n\left(\int_0^tP_s^ku_nds-\int_0^{t+\frac1n} P_s^ku_nds\right)+n\int_0^{\frac 1n}P_s^ku_nds\nonumber\\
  &=&-n\int_t^{t+\frac1n} P_s^ku_nds+n\int_0^{\frac 1n}P_s^ku_nds.\label{equ5}
  \end{eqnarray}
By~(\ref{eq18}), $n\int_t^{t+\frac1n} P_s^ku_nds\leq P_t^ku_n\leq nP_t^kG^kv$ which tends to 0 when $t$ tends to $\infty$ (see Proposition~\ref{prop8}).
Whence, by letting $t$ tends to infinity in~(\ref{equ5}), we obtain  $G^kf_n(x)= n\int_0^{\frac 1n}P_s^ku_nds$ and the proof is finished.

  \end{proof}
 \begin{remark}\label{rem} In virtue of Fatou's lemma, for every $f\in\mathcal B^+(\R^d)$ the function $G^kf$ is lower semi continuous on $\R^d$. Hence, an immediate consequence of the above proposition  is that every excessive function $u$ is lower semi-continuous on $\R^d$ and by  Lemma~\ref{lemme4}, it satisfies
 \begin{equation}\label{eq13}
 M_{x,t}(u)\leq u\quad \mbox{for all }x\in\R^d, t>0.
 \end{equation}
 \end{remark}
 \begin{theorem}\label{balayage}
 The couple $(\R^d, E_{\Delta_k})$ is a balayage space.
 \end{theorem}
 \begin{proof} It follows from~Proposition\ref{prop9} and Lemma~\ref{lemme3} that $E_{\Delta_k}$ is linearly separating.
  Then in view of~\cite[V.2.4]{hansen} the proof will be finished once we have shown that there exist positive functions $u,v\in E_{\Delta_k}\cap C(\R^d)$ such that $\frac{u}{v}\in C_0(\R^d)$. To that end, let $v:=1$ and $u:=\min(G^k(\cdot,0),1)$ which are obviously excessive. It is easy to check from~(\ref{eq7}) that
 $$G^k(\cdot,0)=\frac{1}{2d_k\lambda|\cdot|^{2\lambda}}.$$
 Thus $u,v\in E_{\Delta_k}\cap C(\R^d)$ and $\frac uv\in C_0(\R^d)$.
 \end{proof}
 \section{Harmonic measures}
 For every excessive function $u$ and every open set $V$ let
  \begin{equation}\label{reduit}
H_Vu(x)=\inf\{v(x) : v\in E_{\Delta_k}, v\geq u \mbox{ on } V^c\},\quad x\in\R^d.
 \end{equation}
 Obviously $H_Vu(x)=u(x)$ if $x\in V^c$. In virtue of the general theory of balayage spaces, Theorem~\ref{balayage} yields that for each point  $x\in \R^d$ and $V\Subset \R^d$, there exists a unique probability measure $H_V(x,\cdot)$ on $\R^d$ which is supported by $V^c$ such that for every excessive function $u$
 $$H_Vu(x)=\int_{V^c}u(z)H_V(x,dz).$$
 It is clear that if $x\in V^c$ then $H_U(x,\cdot)$ is  the Dirac measure concentrated at~$x$.
  In the sequel,  we denote
  $$H_Vf(x)=\int_{V^c}f(z) H_V(x,dz)$$
  when the integral makes sense. Obviously
 \begin{equation}\label{equ1}
 H_Uf=f\quad \mbox{ on } U^c.
 \end{equation}
  In the following we collect some useful properties of  $H_V$ (see \cite[Chapter III]{hansen} for more details).
  \begin{proposition}\label{prop6}   Let $f:\R^d\rightarrow \overline \R$ be a Borel function and $V$ be a bounded open set.
\begin{enumerate}
   \item  The function $f$ is excessive if and only if $f$ is lower semi-continuous and for every  bounded open set $U$
 $$H_Uf\leq f\mbox{ on }\R^d.$$
   \item If $f$ is bounded and with compact support on $\R^d$ then $H_Vf\in C(V)$.
   \item
   \begin{equation}\label{eq15}
H_UH_Vf=H_Vf\quad \mbox{ for every }U\Subset V.
\end{equation}
 \end{enumerate}
 \end{proposition}

In this section we shall prove that for every bounded $W$-invariant open set~$V$ and every $x\in V$ the  harmonic measure $H_V(x,\cdot)$  is supported by~$\partial V$.
 \begin{lemma}\label{lemme5}Let  $V$ be a bounded $W$-invariant open subset of $\R^d$ and let $u$ be an excessive function locally bounded on $V$  satisfying $M_{x,t}(u)=u(x)$ for every $x\in V$ and $t>0$ such that $B(x,t)\Subset V$. Then $$H_Uu=u\quad\mbox{ for every  } U\Subset V.$$

 \end{lemma}
 \begin{proof} Let $U\Subset V$. Recall from~(\ref{equ1}) and the first statement of Proposition~\ref{prop6}  that $H_Uu=u$ on $U^c$  and that     $H_Uu\leq u$ on $\R^d$.  So, in order to get equality on~$\R^d$, we  need to  prove that $u\leq H_Uu$ on $U$. In virtue of~(\ref{reduit}), it suffices to show  that $u\leq v$ on $U$ for every  $v\in E_{\Delta_k}$ satisfying $v\geq u$ on $U^c$. Let $v$ be a such function and consider $w=v-u$. Since $v$   is lower semi continuous function on~$\R^d$ (Remark~\ref{rem}) and $u$ is continuous on $V$ (Proposition~\ref{prop1}) we deduce that $w$ is lower semi-continuous on $U$ and that for every $z\in\partial U$
  $$\liminf_{x\rightarrow z}w(x)=v(z)-u(z)\geq 0.$$Furthermore, using~(\ref{eq13}), we obtain that for every $x\in U$ and $t>0$ such that $B(x,t)\Subset U$,
  $$M_{x,t}(w)=M_{x,t}(v)-M_{x,t}(u)\leq v(x)-u(x)= w(x).$$
  Assume first that  $U$ is $W$-invariant  then by Proposition \ref{prop3}, $w\geq 0$ on $U$ and consequently $v\geq u$ on $U$ which implies that
  $H_Uu\geq u$ on $U$. Whence $H_Uu=u$ on $\R^d$. Now we turn to the general case where $U$ is arbitrary . Let $A$ be a $W$-invariant open set such that   $U\Subset A\Subset V$. By the preceding part, $H_{A}u=u$ on~$\R^d$. Whence, using~(\ref{eq15}) we obtain
  $$H_Uu=H_UH_{A}u=H_{A}u=u\quad\mbox{ on }\R^d.$$

 \end{proof}

\par It follows from~(\ref{eq8}) that  for every $y\in \R^d$ the function $G^k(\cdot,y)$ is locally bounded on $\R^d\setminus O(y)$,  where $O(y)$ denotes the orbit of $y$ with respect to the group $W$, i.e.,
$$O(y):=\{wy : w\in W\}.$$
Thus, the above lemma as well as Proposition~\ref{prop112} yield that for a fixed   $x\in\R^d$,
 \begin{equation} \label{lemme2}
 \int_{U^c} G^k(\xi,y)H_U(x,d\xi)=G^k(x,y)\quad \mbox{for all } y\in\R^d \mbox{ and } U\Subset \R^d\setminus O(y).
 \end{equation}
 \begin{lemma}\label{lemme6}
 Let $V$ be a $W$-invariant bounded open set and  $\varphi \in C^\infty_c(\R^d)$ such that $\Delta_k\varphi=0$ on $V$. Then $H_U\varphi=\varphi$ for every $U\Subset V$.
 \end{lemma}
 \begin{proof}
 Using (\ref{eq9}), we write $$\varphi=-G^k(\Delta_k\varphi)=G^kh^--G^kh^+$$ where $h^-=\max(0,-\Delta_k\varphi)$ and $h^+=\max(0,\Delta_k\varphi)$. Clearly, $G^kh^-$ and $G^kh^+$ are excessive (Proposition~\ref{prop9}) and $h^+=h^-=0$ on $V$. Hence, in view of~(\ref{cor2}), for every $y\in V$ and $t>0$ such that $B(y,t)\Subset V$,
    $$M_{y,t}(G^kh^-)=G^kh^-(y)\quad \mbox{and }\quad M_{y,t}(G^kh^+)=G^kh^+(y).$$ This leads, by Lemma~\ref{lemme5}, to
    $$H_U(G^kh^-)=G^kh^-\quad\mbox{ and }\quad H_U(G^kh^+)=G^kh^+$$ and so $H_U\varphi=\varphi$ for every $U\Subset V$.

 For an open subset $U$ of $\R^d$
 it will be convenient to denote by $^W\!\!U$   the smallest $W$-invariant open set containing $U$, i.e.
  $$^W\!\!U:=\bigcup_{w\in W}w(U).$$
    Of course, if $U\Subset A$ for some $W$-invariant open set $A$ then $ \overline{^W\!\!U}\subset A$.

  \begin{proposition}\label{prop5}
  Let $U$ be a bounded open set. For every $x\in U$
      $$\mbox{supp }H_U(x,\cdot)\subset \overline{^W\!\!U}\setminus U,$$
  In particular, if $U$ is $W$-invariant then for every $x\in U,$
  $$\mbox{ supp }H_U(x,\cdot)\subset \partial U.$$
  \end{proposition}
  \begin{proof}
  For every $n\geq 1$ we define
 $$  U_n=\{y\in\R^d : \inf_{x\in ^W\!\!U}|x-y|<\frac1n\}.$$
 Obviously,  $U_n$ is a $W$-invariant  open set, $^W\!\!U\Subset U_n$ and  $U\Subset U_n$ for all $n\geq 1$. Furthermore $\overline{^W\!\!U}=\cap_{n\geq1}U_n.$    Let $n\geq 1$ and  $\varphi\in C^\infty_c(\R^d)$ such that $\varphi=0$ on $U_n$. Then by the above lemma,  $H_U\varphi=\varphi=0$ on $U$. That is for every $x\in U$
    $$\int_{U^c}\varphi(y)H_U(x,dy)=0,$$
    which means that $\mbox{supp }H_U(x,\cdot)\subset \overline U_n$ and consequently  $\mbox{supp }H_U(x,\cdot)\subset \overline {^W\!\!U}$ (because $\overline{^W\!\!U}=\cap_{n\geq 1}U_n.$). Finally recall that the support of  $H_U(x,\cdot)$ is supported by $U^c$ to conclude.
  \end{proof}
\begin{corollary}\label{coro1} Let $U$ be a bounded open set. Then $H_Uf\in C(U)$ provided $f\in \mathcal B_b(^W\!\!U\setminus U)$.
In particular, if $U$ is $W$-invariant then
\begin{equation}\label{eq16}
H_Uf\in C(U)\quad\mbox{for every }f\in\mathcal B_b(\partial U).
\end{equation}
\end{corollary}
\begin{proof} Let $f\in \mathcal B_b(^W\!\!U\setminus U)$.
We may extend $f$ to $\tilde f$ on $\R^d$ by setting $\tilde f=f$ on $^W\!\!U\setminus U$ and $\tilde f=0$ otherwise. So $\tilde f$ is a Borel Bounded function on $\R^d$ with compact  support. Moreover using  the above theorem we see that  $H_Uf=H_U\tilde f$ which is continuous on $U$ by  the second statement of Proposition~\ref{prop6}.
\end{proof}

 \end{proof}

 \section{Dirichlet problem}
\label{sec3}
\par In all this section  $V$  denotes a $W$-invariant bounded open set. A sequence  $(x_n)_n\subset V$ converging to a point $z\in \partial V$ is said  regular with respect to $V$ provided
$$\lim_{n\rightarrow \infty }H_Vf(x_n)=f(z),\quad\mbox{for every }f\in C_c(\R^d) .$$
A point $z\in\partial V$ is called regular if every sequence $(x_n)_n$ on $V$ converging to $z$ is regular. The open set $V$ is called regular if every $z\in\partial V$ is regular.
\par This section is devoted to prove that  for every continuous function $f$on $V$, $f$ is $\Delta_k$-harmonic on $V$ if and only if $H_Uf=f$ for every $U\Subset V$. Furthermore, assuming that  $V$ is regular  we shall show that  the function $H_Vf$ is  the unique solution $u\in C^2(V)\cap C(\overline V)$ to the Dirichlet problem
$$\left\{\begin{array}{rcl}
\Delta_ku&=&0 \quad\textrm{ on~}V,\\
u&=&f\quad\textrm{ on~}\partial V.
\end{array}\right.$$
It will be commode to denote  $G^k_y$  for the Green function $G^k(\cdot,y)$. For every $n\geq 1$ let
\begin{equation}\label{suitee}
V_n=\{z\in V : B(z,\frac1n)\Subset V\}.
\end{equation}
It is clear that $(V_n)_{n\geq 1}$ is an increasing sequence of $W$-invariant open sets satisfying  $V_{n}\Subset V_{n+1}\Subset V$ for every $n\geq1$ and  $V=\bigcup_{n\geq 1} V_n$.

 \begin{proposition}\label{theorem} Let $f$ be a continuous function on $V$. If $f$ is $\Delta_k$-harmonic on~$V$ then $H_Uf=f$ for every $U\Subset V$.
\end{proposition}
\begin{proof}
Assume that $f$ is $\Delta_k$-harmonic on $V$. It follows from the hypoellipticity of the operator~$\Delta_k$ that $f\in C^\infty(V)$. Let $U\Subset V$. Let $n_0\geq 1$ such that $U\Subset V_{n_0}$ and  consider $\varphi\in C^\infty_c(\R^d)$ such that $\varphi=f$ on $V_{n_0}$. Then $\Delta_k\varphi=0$ on $V_{n_0}$ and so, in view of Lemma~\ref{lemme6}, $ H_U\varphi=\varphi$. On the other hand, $H_U\varphi=H_Uf$ on $U$ since  $V_{n_0}$ contains the support of the measure $H_U(x,\cdot)$ for every $x\in U$ (see Proposition~\ref{prop5}). Whence $H_Uf=\varphi=f$ on $U$. The equality on $U^c$ follows from~(\ref{equ1}).
\end{proof}

\begin{lemma}
For each $y\in V$, the function $u:=\lim_{n\rightarrow \infty}H_{V_n}G^k_y$ belongs to  $C_b(V)$ and satisfies
$$H_Uu=u\quad\mbox{ for every }U\Subset V.$$
\end{lemma}
\begin{proof}
 Let $n_0\geq 1$ and $\varepsilon >0$ such that   $B(y,\varepsilon)\Subset V_{n_0}$ and so $^W\!\!B(y,\varepsilon)\Subset V_{n_0}$.  Let $n>n_0$ and let $x\in V_n$.
Then for every $\xi\in \partial V_n$, $|\xi-wy|>\varepsilon$ for all $w\in W$ which implies in virtue of~(\ref{eq8}) that
$$H_{V_n}G^k_y(x)=\int_{\partial V_n} G^k(y,\xi) H_{V_n}(x,d\xi)\leq \frac{\varepsilon^{-2\lambda}}{2d_k\lambda}.$$
So by letting $n$ tends to infinity we easily see that $u$ is bounded on $V$. Let now $U\Subset V$, then there exists $n_1\geq 1$ such that $U\Subset V_n$ for every $n\geq n_1$. Therefore, by~(\ref{eq15}), for every $x\in U$
\begin{eqnarray*}
H_{V_{n_1}}G^k_y(x)-H_Uu(x)&=& H_U(H_{V_{n_1}}G^k_y-u)(x)\\
&=&\int_{U^c} \lim_{n\rightarrow \infty}(H_{V_{n_1}}G^k_y-H_{V_n}G^k_y)(\xi)H_U(x,d\xi).
\end{eqnarray*}
The sequence $(H_{V_n}G^k_y)_{n\geq n_1}$ is  decreasing. Indeed, Since $G^k_y$ is excessive it follows from Proposition~\ref{prop6} that
$H_{V_{n+1}}G^k_y\leq G^k_y$. Applying $H_{V_{n}}$ and  using~(\ref{eq15}),  we get $H_{V_{n+1}} G^k_y\leq H_{V_{n}}G^k_y$ for every $n\geq n_1$. Whence
$(H_{V_{n_1}}G^k_y-H_{V_n}G^k_y)_{n\geq n_1}$ is a non negative increasing sequence. Then by the monotone convergence theorem
\begin{eqnarray*}
H_{V_{n_1}}G^k_y(x)-H_Uu(x)&=& \lim_{n\rightarrow \infty} \int_{U^c}(H_{V_{n_1}}G^k_y-H_{V_n}G^k_y)(\xi)H_U(x,d\xi)\\
&=& \lim_{n\rightarrow \infty}H_U(H_{V_{n_1}}G^k_y-H_{V_n}G^k_y)(x)\\
&=&\lim_{n\rightarrow \infty}(H_{V_{n_1}}G^k_y-H_{V_n}G^k_y)(x)\\
&=&H_{V_{n_1}}G^k_y(x)- u(x).
\end{eqnarray*}
This means that $H_Uu(x)=u(x)$. Hence $H_Uu=u \mbox{ on }U$ and then on $\R^d $ (using~\ref{equ1}).
In particular  for every $n\geq 1$, $H_{V_n}u=u$ on $V_n$. Since  $H_{V_n}u$ is continuous  on $V_n$ by Corollary~\ref{coro1}, it follows that $u$ is continuous  on $V_n$ for every $n$ and then $u$ is continuous  on $V$.
\end{proof}

\begin{lemma}
Let $y\in V$ and $(V_n)_n$ be as in~(\ref{suitee}).
Then
$$H_VG^k_y=\lim_{n\rightarrow\infty }H_{V_n}G^k_y.$$
\end{lemma}
\begin{proof}
 Let us denote $u=\lim_{n\rightarrow \infty}H_{V_n}G^k_y.$ Since $G^k_y$ is excessive it follows from Proposition~\ref{prop6} that
$H_{V}G^k_y\leq G^k_y$. Applying $H_{V_n}$ and  using~(\ref{eq15}),  we get $H_{V} G^k_y\leq H_{V_n}G^k_y$ for every $n\geq n_0$. Let $n$ tends to infinity to obtain
\begin{equation}
H_VG^k_y\leq u.
\end{equation}
To prove the converse equality, we denote $v=H_VG^k_y-u$ and we intend to show that $v\geq 0$ on $V$. Using a general minimum principle of balayage spaces (see \cite[III.4.3]{hansen}) it will be sufficient  to  show that $v$ is lower semi continuous on $V$, $v\geq 0$ on $V^c$, $\inf v(V)>-\infty$, $H_Uv\leq v$ for every $U\Subset V$ and that $\liminf_{n\rightarrow\infty}v(x_n)\geq 0$ for every regular sequence $(x_m)_m$ on $V$.\\
In view of the above lemma, it is clear that $v$ is continuous on $V$, $\inf v(V)>-\infty$ and $H_Uv\leq v$ for every $U\Subset V$. Moreover, for every $n\geq 1$, $H_VG^k_y=G^k_y=H_{V_n}G^k_y$ on $V^c$. This yields that $v=0$ on $V^c$. Finally, let $(x_m)_m$ be a regular sequence on $V$ converging to $z\in\partial V$. Let $f\in C_c(\R^d)$ such that $f=G^k_y$ on $\partial V$. By Proposition~\ref{prop5}, $H_{V}G^k_y=H_{V}f$ on $V$ and then
$$\lim_{m\rightarrow\infty }H_{V}G^k_y(x_m)=\lim_{m\rightarrow\infty }H_{V}f(x_m)=f(z)=G^k_y(z).$$
Furthermore,
$$\lim_{m\rightarrow\infty } u(x_m)\leq \lim_{m\rightarrow\infty }G^k_y(x_m)=G^k_y(z).$$
Whence $\liminf_{n\rightarrow\infty}v(x_n)\geq 0$.  This implies that $v\geq 0$ and finishes the proof.
\end{proof}
\begin{lemma} For every $x,y\in \R^d\setminus~\partial V$
\begin{equation}\label{symetri}
H_VG^k_y(x)=H_VG^k_x(y).
\end{equation}
In particular if $y\in \overline V^c$, then $H_VG^k_y=G^k_y$.
\end{lemma}
\begin{proof} Let $x,y\in\R^d\setminus \partial V$. If $y\in \overline V^c$, then $H_VG^k_x(y)=G^k(x,y)$. Hence, it follows from~(\ref{lemme2}) that $H_VG^k_x(y)=H_VG^k_y(x).$ Now, assume that  $y\in V$. Let $(V_n)_{n\geq1}$  be as in~(\ref{suitee}). Let us consider the function~$u$ defined for every $\eta\in \R^d$ by   $u(\eta):=H_VG^k_\eta(y)$.
Then for every $n\geq 1$ and every $\eta\in V_n$,
$$H_{V_n}u(\eta)=\int_{\partial V_n}\int_{\partial V} G^k(\xi,z)H_V(y,d\xi)H_{V_n}(\eta,dz)=\int_{\partial V} H_{V_n}G^k_\xi(\eta)H_V(y,d\xi).$$
By~(\ref{lemme2}), for every $\xi\in \partial V$,  $H_{V_n}G^k_\xi=G^k_\xi$ on~$\R^d$. Consequently, $H_{V_n}u=u$ on $V_n$ and then, using~(\ref{equ1})
\begin{equation}\label{equ2}
H_{V_n}u=u\quad  \mbox{on }\R^d.
\end{equation}
 Since  $G^k_y$ is excessive we then deduce from Proposition~\ref{prop6} that $u\leq G^k_y$ on~$\R^d$ which implies that for every $n\geq 1$, $H_{V_n}u\leq H_{V_n}G^k_y$ on~$\R^d$. Then, using~(\ref{equ2}), we get $u\leq H_{V_n}G^k_y$.  Letting $n$ tends to $\infty$ we obtain by the above lemma $u\leq H_VG^k_y$ on~$\R^d$. In particular,
$$u(x)=H_VG^k_x(y)\leq H_VG^k_y(x).$$
 Finally, interchange $x$ and $y$ to derive equality.
  \end{proof}

 \begin{proposition} For every continuous function $f\in\partial V$ the function $H_Vf$ is $\Delta_k$-harmonic on $V$.
 \end{proposition}
 \begin{proof} Let $x\in V$ and $t>0$ such that $B(x,t)\Subset V$. In virtue of proposition~\ref{prop1} it suffices to prove that $M_{x,t}(H_Vf)=H_Vf(x)$. First we claim that
   \begin{equation}\label{equ3}
   M_{x,t}(H_VG^k_y)=H_VG^k_y(x)\quad\mbox{for every }y\in \R^d\setminus \partial V.
   \end{equation}
    Indeed, if $y\in \overline V^c$ then by~(\ref{lemme2}), $H_VG^k_y=G^k_y$ and so, in view of Proposition~\ref{prop112}, we get $M_{x,t}(H_VG^k_y)=H_VG^k_y(x)$. Assume now that $y\in V$. Then, in view of~(\ref{symetri}),
 \begin{eqnarray*}
 M_{x,t}(H_VG^k_y)=\int_{V} H_VG^k_y(z)d\sigma_{x,t}^k(z)&=& \int_{V} H_VG^k_z(y)d\sigma_{x,t}^k(z)\\
&=&\int_{\partial V}\int_{ V}G^k(\xi,z)d\sigma_{x,t}^k(z)H_V(y,d\xi)\\
&=&\int_{\partial V}M_{x,t}(G^k_{\xi})H_D(y,d\xi).
\end{eqnarray*}
By Proposition~\ref{prop112}, for every  $\xi\in \partial V$, $M_{x,t} (G^k_\xi)=G^k_\xi(x)$. This leads to  $M_{x,t}(H_VG^k_y)=H_VG^k_y(x)$ and proves the claim.
An immediate consequence of~(\ref{equ3}) together~(\ref{eq3}) is  that for every $\varphi\in C^\infty_c(\R^d)$,
\begin{equation}\label{equ4}
M_{x,t}(H_V(G^k\varphi))=H_V(G^k\varphi)(x).
\end{equation}
Let $(\varphi_n)_n\subset C^\infty_c(\R^d)$  be a sequence converging to $f$ on $\partial V$. Using~(\ref{eq9}) we write $\varphi_n=G^k(-\Delta_k\varphi_n)$ for every $n$  and then it follows from~(\ref{equ4})  that
$$M_{x,t}(H_V(\varphi_n))=H_V(\varphi_n)(x).$$
Finally, let $n$ tends to infinity to finish the proof.
 \end{proof}

The following result is an immediate consequence of the above proposition and Proposition~\ref{theorem}.
  \begin{corollary}
Let $f\in C(V)$. Then $H_Uf=f$   for every  $U\Subset V$ if and only if $f$ is $\Delta_k$-harmonic on~$V$.
\end{corollary}

\begin{theorem}\label{thpbd} Assume that $V$ is regular. Then for every   $f\in~C(\partial V)$, the function $H_Vf$ is  the unique solution $u\in C^2(V)\cap C(\overline V)$ to the Dirichlet problem
$$\left\{\begin{array}{rcl}
\Delta_ku&=&0 \quad\textrm{ on~}V,\\
u&=&f\quad\textrm{ on~}\partial V.
\end{array}\right.$$
\end{theorem}
 \begin{proof} The function $H_Vf$ is $\Delta_k$-harmonic on $V$ by the above corollary. Moreover, $H_Vf=f$ on $\partial V$ and  $H_Vf$ is continuous  on $\partial V$ since $V$ is regular. To prove the uniqueness, let $u,v\in C^2(V)\cap C(\overline V)$ be two solutions to the Dirichlet problem. Then the function~ $h:=u-v$ satisfies $h\in C^2(V)\cap C(\overline V)$,  $\Delta_kh=0$ on~$V$    and  $h=0$ on~$\partial V.$ Hence for every $x\in\R^d$ and every $t>0$ such that $\overline B(x,t)\subset V$, $M_{x,t}(h)=h(x)$ by Proposition~\ref{prop1}.
  Applying Proposition~\ref{prop3} to   $h$ and $-h$, we get $h=0$~on~$V$. This  finishes the proof.
 \end{proof}

\end{document}